\documentclass
{amsart}

\usepackage{amsmath}
\usepackage{amsfonts}
\usepackage{amssymb}
\usepackage{amsmath, amsthm, amssymb, amsfonts}
\usepackage{verbatim}

\usepackage[mathlines, pagewise]{lineno}

\usepackage
{hyperref}
\hypersetup{
  colorlinks  = true,    
  urlcolor    = red,    
  linkcolor   = blue,    
  citecolor   = blue,      
  pdfauthor   = {Hector Alonzo Barriga Acosta},%
  pdfsubject  = {Protocolo},%
  pdftitle    = {Countable Nabla Products},  %
}

\newtheorem{prop}{Proposition}
\newtheorem{thr}[prop]{Theorem}

\newtheorem{lemma}[prop]{Lemma}
\newtheorem{claim}{\rm Claim}[prop]
\theoremstyle{definition}
\newtheorem{de}[prop]{Definition}
\newtheorem{exa}[prop]{Example}

\newtheorem{ques}[prop]{Question}

\newcommand{\om}{\omega}
\newcommand{\baire}{\omega^\omega}

\newcommand{\sub}{\subseteq}

\newcommand{\rest}{\upharpoonright}

\newcommand{\up}[1]{\uparrow\!{#1}}
\newcommand{\down}[1]{\downarrow\!{#1}}

\title{Box and Nabla Products that are $D$-Spaces}
\author[H. A. Barriga-Acosta]{H. A. Barriga-Acosta}
\address{Department of Mathematics,
University of North Carolina at Charlotte, 
Charlotte, NC 28223}
\email{hector.alonsus@gmail.com}

\author[P. M. Gartside]{P. M. Gartside}
\address{Department of Mathematics,
University of Pittsburgh, 
Pittsburgh, PA 15260}
\email{gartside@math.pitt.edu }
\date{\today}


\begin{document}
\maketitle

\begin{abstract}
A space $X$ is $D$ if for every assignment, $U$, of an open neighborhood to each point $x$ in $X$ there is a closed discrete $D$ such that $\bigcup \{U(x) : x \in D\}=X$. 
The box product, $\square X^\om$, is $X^\om$ with topology generated by all $\prod_n U_n$,  where  every  $U_n$ is open. The nabla product, $\nabla X^\om$, is obtained from $\square X^\om$ by quotienting out mod-finite. 
The weight of $X$, $w(X)$, is the minimal size of a base, while $\mathfrak{d}=\mathop{cof} \om^\om$.

It is shown that there are specific compact spaces $X$ such that $\square X^\om$ and $\nabla X^\om$ are not $D$, but in general:

(1) $\square X^\om$ and $\nabla X^\om$ are hereditarily $D$ if $X$ is scattered and either hereditarily paracompact or of finite scattered height, or if $X$ is metrizable (and $w(X)\le \mathfrak{d}$ for $\square X^\om$);

(2) $\nabla X^\om$ is hereditarily $D$ if $X$ is first countable and $w(X)\le \omega_1$, or \emph{consistently} if $X$ is first countable and $|X|\le \mathfrak{c}$, or $w(X)\le \omega_1$; and

(3) $\square X^\om$ is $D$ \emph{consistently} if $X$ is compact and either first countable or $w(X)\le \omega_1$.
\end{abstract}

\section{Introduction}

We investigate property $D$ in box and nabla products. 
Recall that a \emph{neighbornet} or \emph{neighborhood assignment} on a space $X=(X,\tau)$ is a map $U : X \to \tau$ such that $x$ is in $U(x)$ for every $x$ in $X$. We automatically extend a given neighbornet, $U : X \to \tau$, over all subsets of $X$ by taking unions: $U(S)=\bigcup \{ U(x)  : x \in S \}$ for $S \subseteq X$. 
Then a space $X$ is a \emph{$D$-space} if for every neighbornet $U$ there is a closed discrete set $D$ in $X$ such that $U(D) = X$. 

The $D$-space property can be thought of as a covering property.
Indeed, every closed subspace of a $D$-space is \emph{irreducible}: every open cover has an open refinement such that no proper subcollection covers.
But it is unknown - and a famous open problem - whether paracompact or even hereditarily Lindelof spaces are $D$.

The box product problem, on the other hand, is concerned with when box or nabla products are paracompact. 
There are a number of ZFC counter-examples of box and nabla products which are not paracompact. There are various consistent positive results, in which some box and nabla products are paracompact.
But there are few positive results in ZFC, and few consistent counter-examples to complement the consistent theorems.

Thus it is valuable to determine which box or nabla products are $D$. Of particular interest will be those cases where the box or nabla product is known (consistently) to be paracompact. 
We will show that many box and nabla products are $D$-spaces, indeed quite often hereditarily $D$. This will include most cases where we know the box or nabla product is paracompact.
Moreover, these results largely hold in ZFC.

Given spaces $X_n$, for $n$ in $\mathbb{N}$, the \emph{box product}, $\square_n X_n$, is $\prod_n X_n$ with topology generated by open boxes, $\prod_n U_n$,  where for every  $n$ the set  $U_n$ is open.
We write $\square X^\omega$ for $\square_n X_n$ where $X_n=X$ for all $n$. 
Box products of non-compact spaces are typically not paracompact. 
Indeed box products of `large' compacta are also not paracompact. In Section~\ref{sec:cpta} we give two examples of compact spaces, $X$, whose box product, $\square X^\om$, is not $D$.  

In order to show that some box products of compact spaces are paracompact, Kunen \cite{kunen1985box} introduced nabla products. 
Let $=^*$ be  mod finite equivalence on $\prod_n X_n$.
Then the \emph{nabla product}, $\nabla_n X_n$, is the quotient space $\left(\square_n X_n\right)/=^*$. Denote by $q$ the quotient map from $\square_n X_n$ to $\nabla_n X_n$. 
We write $\nabla X^\omega$ for $\nabla_n X_n$ where for every $n$, $X_n=X$. 
Kunen showed that, provided the spaces $X_n$ are all compact, the box product $\square_n X_n$ is paracompact if and only if the nabla product $\nabla_n X_n$ is paracompact. 
Almost every positive result on the paracompactness of a box product,  $\square_n X_n$, requires the $X_n$'s to be compact and is deduced from Kunen's result by showing that $\nabla_n X_n$ is paracompact.  

For example, it is consistent (in various models) that if $X$ is compact and either first countable or of weight no more than $\omega_1$, then $\square X^\om$ is paracompact (because $\nabla X^\om$ is paracompact). 
Here we show, in Section~\ref{ssec:small}, that, in the same models, if $X$ is first countable and $|X|\le \mathfrak{c}$, or if it has weight no more than $\omega_1$ then $\nabla X^\om$ is hereditarily $D$; and we deduce that if $X$ is compact and either first countable or of weight no more than $\omega_1$ then $\square X^\om$ is $D$. 
Moreover, in Section~\ref{ssec:metric}, we go on to show - in ZFC - that $\nabla X^\om$ is hereditarily $D$ for every metrizable $X$, and $\square X^\om$ is hereditarily $D$ when $X$ is metrizable and has weight no more than the \emph{dominant} number, $\mathfrak{d}$, the minimal size of a cofinal famly in $\om^\om$ with the product order. 
Further, $\nabla X^\om$ is hereditarily $D$ if $X$ is first countable and has weight no more than $\omega_1$. 

Kunen himself applied nabla products to show that, under the Continuum Hypothesis, (CH), for every compact scattered space $X$, the box product, $\square X^\om$, is paracompact. 
What can be said about the $D$-property of box or nabla products of compact scattered spaces, is an intriguing open problem.
\begin{ques}
If $X$ is compact and scattered then is $\square X^\om$ a $D$-space (a) in ZFC? or (b) at least under (CH)?
\end{ques}
The authors prove in Sections~\ref{sec:hsct} and~\ref{sec:fsct} that if $X$ is scattered and either is hereditarily paracompact or has finite scattered height then $\square X^\om$ and $\nabla X^\om$ are hereditarily $D$. 
An example is given of scattered spaces, $X_n$, where each $X_n$  has scattered height $n$, such that $\nabla_n X_n$ is not $D$. 

We start the paper, in Section~\ref{sec:pre}, by establishing a $D$-space analogue of Kunen's key result connecting box and nabla products, followed by some tools for showing that spaces are $D$, and a machine to create special non-$D$ spaces.

\section{Preliminaries}\label{sec:pre}

If $U: X \to \tau$ is a neighbornet and $N \subseteq X$, we use $U[N]$ to denote the collection $\{ U(x) : x\in N\}$ and $U(N)$ for $\bigcup_{x\in N} U(x)$. If $U$ and $V$ are both neighbornets then we say $U$ is \emph{below} $V$ if  $U(x) \subseteq V(x)$ for every $x$. 

\subsection{Relating the \texorpdfstring{$D$}{D}-property in Box and Nabla Products}

First let us show that for \emph{compact} spaces, $X_n$, the box product, $\square_n X_n$, is a $D$-space if and only if the corresponding nabla product, $\nabla_n X_n$, is $D$. 

\begin{thr}\label{th:boxD_iff_nablaD}
Suppose $X_n$ is compact, $n\in \om$. Then, $\square_n X_n$ is a $D$-space if and only if $\nabla_n X_n$ is a $D$-space.
\end{thr}
This follows from some general facts about closed maps and the $D$-space property. 
Since the quotient mapping, $q$, from $\square_n X_n$ to $\nabla_n X_n$ is closed, and closed images of $D$-spaces are $D$ (\cite{borges1991study}), the forward implication is clear. 
Also, notice that the fibers $q^{-1}(y)$, $y\in Y$, are $\sigma$-compact, hence $D$-subspaces of $\square_n X_n$ and so we can apply the next lemma to derive the converse.

\begin{lemma}\label{Lemma_D_quotient}
Let $q :X \to Y$ be a closed surjective mapping.
If $Y$ is a $D$-space such that $q^{-1} (y)$ is a $D$-subspace of $X$, for every $y\in Y$, then $X$ is a $D$-space.
\end{lemma}
\begin{proof}
Let $U: X \to \tau_X$ be a neighbornet. 
Every fiber $q^{-1}\{y \}$ is a $D$-space; there is a closed discrete set $D_y$ in $q^{-1}(y)$ (hence, in $X$) such that $q^{-1}(y) \subseteq U(D_y)$. 
For $y\in Y$, define $V: Y \to \tau_Y$ as $V(y) = \{ w\in Y : q^{-1}(w)  \subseteq U(D_y)\}$. 
Observe that $V(y)$ coincides with the set $Y \setminus q[X \setminus U(D_y)]$ which is open since $q$ is closed.
Since $Y$ is a $D$-space, there is a closed discrete set $D'$ in $Y$ such that $Y = V(D')$. 
Set $D = \bigcup_{y\in D'} D_y$. 

\medskip

\noindent {\bf Claim:} $D$ is closed discrete in $X$ and $X = U(D)$.

Fix $x\in X$. There is $y\in D'$ such that $q(x) \in V(y)$. 
Thus, $x \in q^{-1}(q(x)) \subseteq U(D_y) \subseteq U(D)$. This proves $X = U(D)$.

Now, in order to prove $D$ is closed discrete we will see that for every $x\in X$ there is a neighborhood $U_x$ around $x$ such that $|D\cap U_x| \leq 1$. 
Let's note that there is $y\in D'$ and open set $V_x$ around $q(x)$ such that $q(x) \in V_x \subseteq V(y)$ and $|V_x \cap D'| \leq 1$. On the other hand, since $D_{q(x)}$ is closed discrete in $X$, there is a neighborhood $V'_x$ around $x$ such that $|V'_x \cap D_{q(x)} |\leq 1$. Set $U_x = V'_x \cap q^{-1}[V_x]$. Observe that $|D \cap U_x|\leq 1$. The claim is proved, and therefore $X$ is a $D$-space.
\end{proof}

Note that this result on preservation of $D$ under inverse images of closed maps is optimal:
closed subsets of $D$-spaces are $D$, so the fibers must be $D$; and if $X$ is any space which is not $D$ then the trivial quotient map to the one point space is closed and onto a $D$-space.

Below we show that many nabla products are not just $D$ but \emph{hereditarily} $D$. 
However the natural `hereditary $D$' versions of Theorem~\ref{th:boxD_iff_nablaD} and Lemma~\ref{Lemma_D_quotient} do not hold. 
Further the restriction in Theorem~\ref{th:boxD_iff_nablaD} to compact spaces is necessary. 
Indeed,  $\square (\omega_1)^\omega$ is not $D$ (because $\omega_1$ is not $D$ and embeds as a closed subspace in $\square (\omega_1)^\omega$), so $\square (\omega_1+1)^\omega$ is not hereditarily $D$, while $\nabla (\omega_1+1)^\omega$ is hereditarily $D$ (see Example~\ref{ex:om1+1}).

\subsection{Showing a Space is \texorpdfstring{$D$}{D}}
Now we turn to the problem of showing that a space is $D$. 
To do so, we need to build suitable closed discrete sets, and here we outline some machinery of \cite{fleissner2001d} which simplifies that process, and then apply that to spaces whose topologies come from partial orders.
\medskip

Let $U$ be a neighbornet on $X$ and $D \subseteq X$. We say $D$ is {\it $U$-sticky} if $D$ is closed discrete and $\forall x \in X  ( U(x) \cap D \neq \emptyset \to x \in U(D))$.
For a neighbornet $U$ on $X$, let $D(U) = \{D \subseteq  X:D \; \; \text{is } U \text{-sticky} \}$. For each $D,D' \in  D(U)$ we say $D \leq_U D'$ if and only if $(1)$ $D \subseteq D'$, and $(2)$ $(D'\setminus D) \cap U(D) = \emptyset$.

\begin{thr}[\cite{fleissner2001d}]\label{Thr_Equivalence_D}
The following are equivalent for a topological spaces $X$.
\begin{enumerate}
    \item For every neighbornet $U$ of $X$, for every $D \in D(U)$ and for every $x \in X$ there exists $D' \in D(U)$ such that $x \in U(D')$ and $D\leq_U D'$.
    \item $X$ is a $D$-space.
\end{enumerate}{}
\end{thr}

In both the definition and the above equivalence we only need to deal with neighbornets below some fixed neighbornet, or indeed, just any  family of neighbornets cofinal in the `below' order. 
Given a subspace $A$ of $X$ and a neighbornet $U$ on $X$ then there is a natural induced neighbornet $U_A$ on $A$ given by $U_A(x)=U(x) \cap A$.


\begin{de} Let $T$ be a neighbornet of a space $X$. 
A subset $S$ of  $X$ is \emph{super-sticky below $T$} if $S$ is closed discrete and for every  neighbornet $U$ on $X$ below $T$, if $U(x) \cap S \ne \emptyset$ then $x \in S$.
\end{de}
We frequently omit the `below $T$' in `super-sticky below $T$' when it is clear from context.

\begin{lemma}\label{l:ss_sub}
Let $A$ be any subspace of $X$. If $S$ is super-sticky below $T$ in $X$ then $S_A=S \cap A$ is super-sticky below $T_A$ in $A$
\end{lemma}
\begin{proof}
Certainly $S_A$ is closed, discrete in $A$. Take any $U'$ a neighbornet on $A$ below $T_A$. Then there is a neighbornet, $U$, on $X$ below $T$, such that $U(x) \cap A = U'(x)$ for all $x$ in $A$ (so $U'=U_A$). Now if $x$ is in $A$ and $U'(x)$ meets $S_A$ then $U(x)$ meets $S$, so $x \in S$. That is, $x \in S_A$. 
\end{proof}

Here is a convenient way of finding things $\leq_U$-above a given $D$.
\begin{lemma}\label{l:ss_grow}
If $U$ is a neighbornet below some $T$, $D$ is  $U$-sticky and $E$ super-sticky below $T$ then $D'=D \cup (E \setminus U(D))$ is $U$-sticky, and hence $D \leq_U D'$. 
\end{lemma}
\begin{proof}
Clearly $D'$ is closed discrete, $D \subseteq D'$, and $(D'\setminus D) \cap U(D) = \emptyset$. So it suffices to check that $D'$ is $U$-sticky.
Take any $x$ such that $U(x)$ meets $D'$. Since $U(D') \supseteq U(D)$, if $x \in U(D)$ then we are done -- and this happens if $U(x)$ meets $D$ (as $D$ is $U$-sticky). 
So we can suppose $x \notin U(D)$ and $U(x)$ meets $E$. Then, as $E$ is super-sticky, $x \in E$. But $x$ is not in $U(D)$, so $x$ is in $U(E \setminus U(D)) \subseteq U(D')$, as required.
\end{proof}

Combining Theorem~\ref{Thr_Equivalence_D}, Lemmas~\ref{l:ss_sub} and~\ref{l:ss_grow} we deduce:
\begin{prop}\label{pr:ss_cover_implies_D}
Every space which has a cover by super-sticky sets all below a fixed neighbornet is hereditarily D.
\hfill $\square$
\end{prop}

Finally we look at spaces with a partial order.

\begin{de}\label{def:a}
Let $X$ be a space.  We say that a  partial order, $\preceq$, on $X$ is \emph{topological} if for every $x$ in $X$ the down-set, $\down{x}=\{ y : y \preceq x\}$, is a neighborhood of $x$.
\end{de}
Observe that a partial order $\preceq$ is topological if and only if $T(x)=\left(\down{x}\right)^\circ$ (the interior of $\down{x}$) is a neighbornet. Also note that if $\preceq$ is topological on $X$ and $A$ is a subspace then $\preceq$ restricted to $A$ is topological on $A$.

\begin{lemma}\label{l:upsets}
Let $X$ be a space with topological partial order, $\preceq$. Then, for every $x$ in $X$, the  up-set, $\up{x} = \{y : x \preceq y\}$ is closed and super-sticky below $T(x)=\left(\down{x}\right)^\circ$. 
\end{lemma}
\begin{proof}
The claim follows immediately from the observation that  for any $y$ in $X$ we have:  \ $\down{y} \ \ \cap \up{x} \ne \emptyset \iff x \preceq y \iff y \in \ \up{x}$.
\end{proof}

\begin{prop}\label{pr:order_implies_hD}
Let $X$ be a space with topological partial order $\preceq$. 

(1) If every up-set $\up{x}$ is $D$ then $X$ is $D$.

(2) If every up-set $\up{x}$ is hereditarily $D$ then $X$ is hereditarily $D$.
\end{prop}
\begin{proof} We prove (1) by checking the first clause of Theorem~\ref{Thr_Equivalence_D}. 
Take any neighbornet $U$ on $X$ below $T(x)=\left(\down{x}\right)^\circ$, and any $U$-sticky set $D$. Fix $x\in X$. We need to find $D'$ which is closed discrete, $D ' \supseteq D$, $(D'\setminus D) \cap U(D) = \emptyset$, $U$-sticky and $x \in U(D')$.

We may assume $x \notin U(D)$ (otherwise take $D'=D$). By hypothesis, $\up{x}$ is a $D$-space, so is its closed subspace $\up{x} \setminus U(D)$. There is a closed discrete $E$ contained in $\up{x} \setminus U(D)$ such that $U(E) \supseteq \up{x} \setminus U(D)$.
Let $D'=D \cup E$. Note $E=D'\setminus D$. Clearly $D'$ is closed discrete, contains $D$, $U(D')$ contains $x$ (recall, $x \notin U(D)$), and $(D'\setminus D) \cap U(D) = \emptyset$. 

It remains to show $D'$ is $U$-sticky. Take any $y \notin U(D')$. Need to show $U(y) \cap D' = \emptyset$. But $U(D')=U(D) \cup U(E)$, so $y \notin \up{x}$. 
By Lemma~\ref{l:upsets}, as $y \notin \up{x}$ we have $U(y) \cap \up{x} = \emptyset$, so $U(y) \cap E$ is empty.
Further, $D$ is $U$-sticky, so $U(y) \cap D = \emptyset$, since $y \notin U(D) \subseteq U(D')$. Thus $U(y) \cap D'$ is empty, as required. 

\medskip

Now we deduce (2). Let $A$ be a subspace of $X$. Then   $\preceq$ restricted to $A$ is topological. Further up-sets in this subspace have the form $\up{x} \cap A$, which is $D$ by hypothesis.
So we can apply part (1) to deduce $A$ is $D$.
\end{proof}

\subsection{A Machine to Make Spaces that are \texorpdfstring{$P$}{P} but  Not \texorpdfstring{$D$}{D}}


We will show that certain box, $\square_n X_n$, and nabla products, $\nabla_n X_n$ are not $D$. The next lemma allows us to do this directly through a space $X$ containing the $X_n$'s when they form a $\subseteq$-chain, rather than the nabla and box products. 
Denote by $X_\delta$ the $G_\delta$-modification of $X$ ($G_\delta$-sets in $X$ are open in $X_\delta$).

\begin{lemma}\label{l:embed_in_nabla}
If $X$ is the union of $X_n$'s such that $X_n \subseteq X_{n+1}$, $n\in \omega$,  then $X_\delta$ embeds as a closed set in $\nabla_n X_n$.
\end{lemma}
\begin{proof}
Fix some $x_0$ in $X_0$. For $x$ in $X$ set $n(x)$ to be the first $n$ such that $x$ is in $X_n$, and then define $e_x: \omega \to \square_n X_n$ via: $e_x(i)$ is $x_0$ if $i<n(x)$ and is $x$ for all $i\ge n(x)$. Now define $e:X_\delta \to \nabla_n X_n$ by $e(x)=[e_x]$.
We show that $e$ is an embedding. 
It is clear that $e$ is a one-to-one mapping. 

First we check $e$ is continuous. To this end, take any $x$ in $X_\delta$ and basic open $\nabla_n (U_n \cap X_n)$ (where each $U_n$ is open in $X$) containing $e(x)$. 
This means there is an $N\ge n(x)$ such that if $n\ge N$ then $x=e_x(n) \in U_n$.
Let $V=\bigcap_{n \ge N} U_n$. 
Then $x$ is in $V$, which is open in $X_\delta$, 
and for every $y$ in $V$, $e(y)$ is in $\nabla_n (U_n \cap X_n)$ because, $e_y$ is in $e(y)$ and for $n \ge \mathop{max}\{ N, n(y)\}$ we have $y=e_y(n) \in V \subseteq U_n \cap X_n$.

Now we show $e$ is open from $X_\delta$ into $e[X]$. To see this, take any non-empty open set $V$ of $X_\delta$. Then we can find a decreasing sequence, $(U_n)_n$, of open sets in $X$ such that $V=\bigcap_n U_n$ and for all $n$ have $U_n \cap X_n \ne \emptyset$. Now note
$e(\bigcap_n U_n)=\{[e_y] : \forall n \, (y \in U_n)\} = \nabla_n (U_n \cap X_n) \cap e[X]$.

It remains to check that $e[X]$ is closed in $\nabla_n X_n$. Take any $f$ in $\square_n X_n$.
For each $n$ let $U_n=X_n \setminus \{ f(i) : i<n$ and $f(i) \ne f(n)\}$. 
Then $[f]$ is in the basic open set $U=\nabla_n U_n$, and if $[f]$ is not in $e[X]$ then $U$ is disjoint from $e[X]$. 
\end{proof}

Now we encapsulate a `machine' for building a $P$-space $X$ (and so $X_\delta=X$, avoiding the need to work with two topologies) which is the increasing union of open, nice subspaces, $X_n$, but which is not $D$, from a `base' space $B$.
This machine is essentially the ideas from example $\Gamma$ of \cite{vDW} with cardinals moved up from $\aleph_0$ to $\aleph_1$.

\begin{de}
A space $X$ is {\it $\kappa$-metrizable} if it has an open base $\mathcal{B} = \{ N(x,\alpha) : \alpha < \kappa, \; x\in X \}$ so that $\{ N(x,\alpha) : \alpha < \kappa \}$ is a neighborhood base at $x$, and given two points $x, y$ and two ordinals $\alpha \leq \beta < \kappa $ then (i) if $y\in N(x,\alpha)$ then $N(y,\beta) \sub N(x,\alpha)$; and (ii) if $y \notin N(x,\alpha)$ then $N(y,\beta) \cap N(x,\alpha) = \emptyset$.
\end{de}

\begin{thr}\label{thr:ex_machine}
Let $B$ be $\om_1$-metrizable but not metrizable, such that $2^{w(B)} \le |B|$.
Then $B$ can be partitioned, $B=\bigcup_n L_n$, so that every closed subset of $B$ of size $|B|$ meets every $L_n$.
Further, there is a topology, $\tau$, on the set $B$, refining the given topology so that $X=(B,\tau)$ has the following properties: $X$ is \emph{not} a $D$-space, however it is a $P$-space,  which is the increasing union of open subspaces $X_n=\bigcup_{k \le n} L_k$, where each $X_n$ is scattered of height $n$, and is locally both of size $\le \om_1$ and Lindelof.
\end{thr}
\begin{proof} 
Let $\kappa=w(B)$. As $B$ is not metrizable, $\kappa$ is uncountable. Note that as  $2^{w(B)} \le |B|$, in fact $|B|=2^\kappa$. For each $y$ in $B$ fix a neighborhood base, $N(y,\alpha)$ for $\alpha<\om_1$, witnessing $\om_1$-metrizability. 
We say that a subset $S$ of $B$ is \emph{dense in itself} if every neighborhood of every point of $S$ contains at least $\om_1$ elements of $S$. 
Note that if $S$ is a subset of $B$ of cardinality $>\kappa$ then it contains a closed in $S$, dense in itself subset $S'$ of the same size, $|S'|=|S|$.
To see this let $S'=S \setminus \bigcup \{ U_y : U_y$ is an open neighborhood of a $y$ in $S$ such that $U_y \cap S$ is countable$\}$. Observe that if $S'$ is non-empty, it is dense in itself and closed in $S$. To see $S'$ is non-empty note that the weight of $B$ is $\kappa$ the collection of open sets has a subcover of size no more than $\kappa$, the union of which meets $S$ in no more than $\kappa$ points.
Hence we may assume that $B$ is dense in itself. Otherwise, let $B'$ be a dense in itself subset of $B$ of the same size as $B$. Repeat the construction below on $B'$, and then add the points of $B \setminus B'$ as isolated points.

 List all subsets of $B$ which are: dense in themselves, have  size $\kappa$, and which have closure of size $2^{\kappa}$, as 
 $\{ K_\gamma : \gamma < 2^{\kappa} \}$. %
Recursively construct points $p(\alpha, n)\in \overline{K}_\alpha \setminus \{ p(\beta, k) : \beta \leq \alpha \text{ and } k < n \}$. 
For $n>1$, set $L_n = \{ p(\alpha, n): \alpha < 2^{\kappa} \}$ and $L_1 = B \setminus \bigcup_{n > 1} L_n$.
Take any closed subset, $C$, of $B$ of size $2^\kappa$. Passing to a closed subset of the same size, we can suppose $C$ is dense in itself. As the weight of $B$ is $\kappa$, $C$ contains some $K_\gamma$. Now, for each $n$, the point $p(\gamma,n)$ witnesses that $C$ meets $L_n$.

Now we define the topology $\tau$ on $B = \bigcup_n L_n$ (to get $X=(B,\tau)$). 
For every $\gamma$ and $n>1$, let $y=p(\gamma,n)$, 
and choose distinct (here we use the fact that $B$ is dense in itself) points $s(y,\alpha)\in N(y,\alpha)$, for $\alpha < \omega_1$, where the $s(y,\alpha)$'s are additionally required to be in $K_\gamma$ if $K_\gamma$  is a subset of $L_{n-1}$. 
The topology $\tau$ is specified by constructing a neighborhood base $\{ G(y, \alpha) : \alpha < \omega_1 \}$ for $y \in L_n$, recursively on $n$, as follows.
For $y\in L_1$, define $G(y, \alpha) = \{y\}$, for all $\alpha< \omega_1$. 
Now suppose $\{ G(y,\alpha) : \alpha< \omega_1 \}$ has been defined for $y\in \bigcup_{k<n} L_k$. Then for $y\in L_n$ define $G(y, \alpha) = \{ y \} \cup \bigcup_{\beta > \alpha} G(s(y,\beta),\beta)$, $\alpha < \omega_1$.

Set $X_n = \bigcup_{k \le n} L_k$. We check some properties of $X$ and the $X_n$'s: (1) $X$ is a $P$-space, (2) $X$ is scattered with scattered height $\om$, and $X_n=\{ x \in X : \mathop{ht}(x)\le n\}$, (3) locally $X$ is of size $\le \om_1$ and Lindelof, (4) $X$ is Hausdorff and zero-dimensional, hence $T_3$, and (5) $X$ is not $D$. 
Easy inductive arguments establish (1)-(3). 
From (2) it follows that each $X_n$ is open. 
Clearly the topology on $X$ refines that on $B$, hence $X$ is $T_2$. Since it is a $P$-space and locally Lindelof we see the basic neighborhoods are closed (hence clopen), and so $X$ is zero-dimensional (and $T_3$).
Towards (5), define the neighbornet $U$ on $X$ by $U(y)=G(y,0)$. 
Suppose $D$ is a subset of $X$ such that $U(D)=X$.
We show $D$ is not closed discrete. Hence $X$ is not a $D$-space.

First note that as each $U(y)$ has size $\om_1$ but $X$ has size $2^{\kappa}$, and $\kappa$ is uncountable, we must have $|D|=2^{\kappa}$, and thus for some $n>1$ we have $D'=D \cap L_{n-1}$ also of size $2^{\kappa}$ (note $\mathop{cf}(2^\mu) > \mu$ for all infinite $\mu$).
We show that the closure of $D'$ in $X$ meets $L_n$, and so neither $D'$ nor $D$ is closed discrete. 
As observed above, passing to a subset, we can assume $D'$ is dense in itself (in $B$). 
Since the weight of $B$ is $\kappa$, there is a dense (in $B$) subset $K$ of $D'$ which has size $\kappa$. As $D'$ is dense in itself, so is $K$. Also observe that the closure in $B$ of $K$ contains $D'$, and so has size $2^{\kappa}$.
Hence $K=K_\gamma$ for some $\gamma$. 
Let $y=p(\gamma,n)$. Then $y$ is in $L_n$, and - by construction - every basic neighborhood, $G(y,\alpha)$ of $y$, in $X$, meets $K \subseteq D'$. In other words, $y$ is in $\overline{D'} \cap L_n$ -- as required.
\end{proof}

\section{General Compacta}\label{sec:cpta}
We show here that the box product of compact spaces need not be a $D$-space.
We give two examples which use very different techniques. The first, a consistent example built via our `P not D' machine, has  weight $\aleph_2$, which is minimal. The second, which exists in ZFC, has weight $\mathfrak{c}^+$. This example is related to arguments of Scott and van Douwen (see \cite{van1980covering}) showing that in general box products of compact spaces need not be normal (or countably paracompact, or other weak covering and separation properties).

Towards our first example we review real-compactness, measurable cardinals, and the connection between them. 
A space $X$ is \emph{real-compact} if it embeds as a closed subset in some $\mathbb{R}^\kappa$. We can take $\kappa =|C(X)|$. Clearly a product of real-compact spaces is real-compact. A paracompact space is real-compact
provided it has no closed discrete subsets of size the first measurable cardinal.
Measurable cardinals are regular and strong limits. So they are indeed `large' - far above $\beth_{\om_1}$, for example. Below, when we say a set is `small' we mean of size strictly less than the first measurable.
(The following example is inspired by an example of a Lindelof $D$-space whose product with a separable metrizable space is not $D$, which uses a trick of Pzymusinski.)

\begin{exa}
Assuming ($2^\om=\om_1$ and $2^{\om_1}=\om_2$),  the box product, $\square (I^{\om_2})^\om$, of compacta of weight $\om_2$ is not $D$.
\end{exa}
\begin{proof}
Let $B=2^{\om_1}$ with the countable box product topology. Then $B$ is $\om_1$-metrizable. Note that as $2^\om=\om_1$, the weight of $B$ is $\om_1$, and so $2^{w(B)} \le |B|$. 
Hence we can apply Theorem~\ref{thr:ex_machine} to get a topology $\tau$ on $B$ refining the given topology, so that $X=(B,\tau)$ is the increasing union of open subspaces, $X_n$, and $X$ is a $P$-space, locally (Lindelof and size $\le \om_1$), and not $D$.

Let $Z=\bigcup_{n>1} X_n$. Note that $Z$ has exactly the same properties as $X$.
Let $Y=Z \cup X_1$ topologized so that $Z$ is an open subspace with the topology above, and points in $X_1$ get neighborhoods as points of $B$.
This topology is Hausdorff, zero-dimensional and $P$.

We check that $Y$ is $\om_1$-Lindelof. Take any open cover $\mathcal{U}$ of $Y$ by basic open sets. Since the weight of $B$, and so $X_1$ in $Y$, is $\om_1$, there is an $\om_1$-sized subcollection whose union, $U$, covers $X_1$. 
Then $C=Y \setminus U$ is a closed subset of $B$ contained in $Z$. If $C$ has size $\le \om_1$ then we can obviously cover it with $\le \om_1$ more elements of $\mathcal{U}$, and so are done in this case. 
Otherwise, $|C|>\om_1$, so - by the second cardinal arithmetic hypothesis - $|C|=2^{\om_1}=|B|$.
But Theorem~\ref{thr:ex_machine} states that $C$ must meet $L_1$, which is contained in $X_1$ - contradiction.

Now let $W$ be the product of $Y$ with $Z$ as a subspace of $B$. It is a $P$-space. 
Observe that the diagonal set, $\Delta_Z=\{(z,z) : z \in Z\}$ is a closed subset of $W$ and it has the topology that comes from $X$ (this topology is finer than that coming as a subspace of $B$). Hence $W$ is not $D$. 
However, $W$ is real-compact because: (1) $Y$ is real-compact ($\om_1$-Lindelof and $P$-space implies paracompact, and $Y$ is `small'), (2) $Z$ is real-compact (paracompact and `small') and (3) the product of two real-compact spaces is real-compact.

The weight of $W$ is $2^{\om_1}$, so $|C(W)| \le 2^{\om_1}$. Hence $W$ embeds as a closed subset in $\mathbb{R}^{2^{\om_1}}$, or equivalently, in $(0,1)^{2^{\om_1}}$. But note that $\left((0,1)^{2^{\om_1}}\right)_\delta$ is a \emph{closed} subset of $\left([0,1]^{2^{\om_1}}\right)_\delta$.
Hence, $W=W_\delta$ embeds as a closed set in $\left(I^{2^{\om_1}}\right)_\delta$, which embeds as a closed set in $\nabla \left(I^{2^{\om_1}}\right)^\om$. 
Which means this last space is not $D$. Recalling that $2^{\om_1}=\om_2$, and $\nabla \left(I^{2^{\om_1}}\right)^\om$ is $D$ if and only if $\square \left(I^{2^{\om_1}}\right)^\om$ is $D$, we are done.
\end{proof}

Now for our second example. 
\begin{exa} Let $K=\{0,1\}^{\mathfrak{c}^+}$, a compact space of weight $\mathfrak{c}^+$. 
The box product $\square K^\om$ is not $D$.
\end{exa}
\begin{proof} 
First note that $K$ with the $G_\delta$-topology, $K_\delta$,  contains a closed discrete subset of size continuum, and embeds as a closed set in $\square K^\om$.
Next, $K$ is homeomorphic to $K^{\mathfrak{c}^+}$, and so $K_\delta$ is homeomorphic to $(K_\delta)^{\mathfrak{c}^+}$ with the countable box product topology. 
Hence $D(\mathfrak{c})^{\mathfrak{c}^+}$, with the countable box product topology, embeds as a closed set in $\square K^\om$. 
(This argument is from \cite{van1980covering}.) 
But now this latter space is not $D$ by Proposition~\ref{P:compacta-non-D}.
\end{proof}

It remains to show that $D(\mathfrak{c})^{\mathfrak{c}^+}$, with the countable box product topology,  is not $D$.  
In \cite{hirata2015d} Hirata and Yajima proved that 
the space $\omega^{\omega_1}$, with the usual product topology, is not $D$.
Here we lift the cardinals. The argument has some subtleties, and 
in our proof we use elementary submodels that are closed under $\omega$-sequences. 

\begin{prop}\label{P:compacta-non-D}
Let $\kappa \leq \mathfrak{c}$ be a cardinal and let $\lambda > \mathfrak{c}$ be a regular cardinal. Then the space $D(\kappa)^{\lambda}$ with countable boxes is not $D$.
\end{prop}
\begin{proof} 
Write $\pi_\gamma$ for the projection to the $\gamma$-th coordinate of $D(\kappa)^{\lambda}$, so $\pi_\gamma(x)=x(\gamma)$. 
Since $D(\kappa)$ has the discrete topology, in the countable box topology, basic open neighborhoods of an $x$ in $D(\kappa)^{\lambda}$ have the form $N(x, C) = \bigcap_{\gamma \in C} \pi^{-1}_\gamma \{x(\gamma)\}$ for countable subsets $C$ of $\lambda$.
Denote by $c_0$ the constant $0$-valued function  in $D(\kappa)^{\lambda}$.

For each $\alpha$ in $\lambda$, define the subset $F_\alpha$ of $D(\kappa)^{\lambda}$ as: $x$ is in $F_\alpha$ if and only if (i) $x\rest \alpha$ is a one-to-one function in $(D(\kappa)\setminus \{ 0 \} )^{\alpha}$, and (ii) $x(\beta) = 0 $, for all $\beta \geq \alpha$.
Observe that the $F_\alpha$'s are pairwise disjoint.
Next we verify that $F = \bigcup_{\alpha \in \lambda} F_\alpha$ is closed in $D(\kappa)^{\lambda}$. 
To see this take any $y$ not in $F$. Then  either (a) there are $\alpha < \beta$ such that $y(\alpha) = 0 < y(\beta)$, or (b) there is a minimal $\alpha$ so that for all $\beta \geq \alpha$, $y(\beta) = 0$ and $y\rest \alpha$ is not one-to-one.
For case (a) the open neighborhood $\pi_{\alpha}^{-1}(y(\alpha)) \cap \pi_{\beta}^{-1} (y(\beta))$ of $y$ is disjoint from  $F$, as required. For case (b) as $y\rest \alpha$ is not one-to-one there are $\beta, \gamma < \alpha$ such that $y(\beta) = y(\gamma)$. Then all elements of the open set  $\pi_{\beta}^{-1}(y(\beta)) \cap \pi_{\gamma}^{-1} (y(\gamma))$ are not one-to-one in $D(\kappa)^{\alpha}$ either, hence disjoint from $F$.

As $F$ is closed in $D(\kappa)^{\alpha}$, the latter space is not $D$ if the former is not $D$. 
We show that $F$ is not a $D$-space.
To this end, define the neighbornet $V$ on $F$, by $V(x) = \pi_\alpha^{-1} \{0\} \cap F$, for $x\in F_\alpha$. Observe that $V(x)=\bigcup \{F_\gamma : \gamma \le \alpha\}$. 
Take any $D\subseteq F$ such that $F = V(D)$. From the observation it is immediate that $\{\alpha : D \cap F_\alpha \ne \emptyset\}$ is cofinal in $\lambda$. We show that $D$ is not closed discrete with the use of elementary submodels.
\bigskip

Let $M$ be an elementary submodel of $H(\lambda)$ of size the continuum, $\mathfrak{c}$, and closed under $\omega$-sequences with $\langle F_\alpha : \alpha < \lambda \rangle, D, V \in M$. Let $\delta = \sup M \cap \lambda$. Fix any $d\in D \cap F_\alpha$ for some $\alpha > \delta$. 

\begin{claim}
For every countable set of coordinates $E \subseteq \delta$, there is $a \in M\cap D$ such that $d\rest E = a \rest E$ and $a\in F_\beta$ for some $\beta < \delta$.
\end{claim}
Note that $E \subseteq M$, hence $E\in M$ since $M$ is closed under $\omega$-sequences. Also, denote by $t$ the function $d\rest E$ and note that $t$ is in $M$ (we don't require that $d$ belongs to $M$) as $M$ contains every countable partial function with domain a countable set of $\delta$ and range in $D(\kappa)$ (i.e. $D(\kappa)^{E} \subseteq M$ for each $E\in [\delta]^\omega$). 

The model $H(\lambda)$ satisfies the following formula ``there exists an element $d$ in $D \cap \bigcap_{\gamma \in E} \pi_\gamma^{-1} (t(\gamma))$''. 
Since all required parameters are in $M$ ($\pi_\gamma$ is in $M$ as long as $\gamma $ in $M$), this formula con be reflected, so there is $a \in M \cap D \cap \pi_{\gamma \in E}^{-1} (t(\gamma))$.
Also the statement ``there is $\alpha$ such that $a\in F_\alpha$'' can be reflected, thus $\alpha \in M$. This implies $\alpha < \delta$. This concludes the claim.

\begin{claim}
The model $M$ satisfies ``any neighborhood around $d$ has at least two elements of $D$''.
\end{claim}
Take a countable set of coordinates $E \subseteq \delta$ and let $t = d \rest E$. 
By the preceding claim, there is $a \in M\cap D \cap \bigcap_{\gamma \in E} \pi^{-1} (t(\gamma))$ with $a\in F_\alpha$ for some $\alpha < \delta$. Now take $\alpha' > \max \{ E, \alpha \}$ in $M$. Define $E' = E \cup \{ \alpha' \}$ and $t' = t \cup \{(\alpha', d(\alpha'))\}$. Then there is $b \in M\cap D \cap \bigcap_{\gamma \in E'} \pi^{-1} (t(\gamma))$ such that $b\in F_\beta$ for some $\beta < \delta$. 
Then $M$ satisfies that the open neighborhood $\bigcap_{\gamma \in E} \pi^{-1} (t(\gamma))$ contains at least two elements of $D$. The claim is finished.
\medskip

The claim implies that $D$ is not closed discrete. Hence $F$, and thus $D(\kappa)^{\lambda}$ with the countable box topology, is not $D$.
\end{proof}


\section{Hereditarily Paracompact Scattered}\label{sec:hsct}

Let $X$ be a scattered space. Denote by $\mathop{I}(X)$ the set of isolated points of $X$.
Define recursively, $X^{(0)}=X$, $X^{(\alpha+1)}=X^{(\alpha)} \setminus \mathop{I}(X^{(\alpha)})$ and $X^{(\lambda)}=\bigcap_{\alpha < \lambda} X^{(\alpha)}$ for limit $\lambda$. 
Define the \emph{scattered height} of $X$ by $\mathop{ht}(X)=\sup \{ \alpha : X^{(\alpha)} \ne \emptyset\}$.
For $x$ in $X$ define the \emph{scattered height} of $x$ in $X$ by $\mathop{ht}(x,X)=\sup \{ \alpha : x \in X^{(\alpha)}\}$.
For each $x$ in $X$ the set $W_x=\{x\} \cup \{ y : \mathop{ht}(y,X) < \mathop{ht}(x,X)\}$ is an open neighborhood of $x$, we call it the \emph{canonical} neighborhood of $x$.
\medskip

A space $X$ is \emph{ultraparacompact} if every open cover of $X$ has a pairwise disjoint open refinement. 
In \cite{van1980covering} E. van Douwen proved that a paracompact zero-dimensional space $X$ is ultraparacompact. His proof is simple and worth repeating.
Every open cover $\mathcal{U}$ of $X$ has a locally finite refinement $\mathcal{V}$ by clopen sets. Enumerate $\mathcal{V} = \{ V_\alpha : \alpha < \kappa \}$. Then $\mathcal{W} = \{ V_\alpha \setminus \bigcup_{\beta < \alpha} V_\beta : \alpha <\kappa \}$ is a disjoint refinement of $\mathcal{V}$ covering $X$. Since $\mathcal{V}$ consist of clopen sets and it is locally finite, $\mathcal{W}$ consist of open sets (in other words, the initial unions $\bigcup_{\beta < \alpha} V_\beta$ are closed).

\begin{lemma}
For every hereditarily paracompact scattered space $X$ there is a partial order $\preceq_X$ such that for all $x$ in $X$: the down-set, $\down{x}$, is open and contained in the canonical neighborhood, $W_x$, of $x$, while the  up-set, $\up{x}$, is finite.
\end{lemma}
\begin{proof} First note that scattered spaces are zero-dimensional, so a hereditarily paracompact scattered space is hereditarily ultraparacompact. 
Now we prove the claim by induction on the scattered height of $X$. If $X$ is discrete (scattered height zero) then for $\preceq_X$ take equality.

Suppose $\mathop{ht}(X)=\lambda$ is a limit. Then let $\mathcal{W}=\{W_x : x \in X\}$. This is an open cover of $X$ by sets each with scattered height strictly less than $\lambda$. 
As $X$ is ultraparacompact there is a clopen partition of $X$, say $X=\bigoplus \{X_s : s \in S\}$, where each $X_s$ is hereditarily paracompact, and scattered of height strictly less than $X$.
Inductively, then, there are partial orders, $\preceq_{X_s}=\preceq_s$ on each $X_s$ as in the claim. Now note that if $\preceq_X$ is the union of  all the $\preceq_s$ (so $x \preceq_X y$ if and only if $x,y$ are both in some $X_s$ and $x \preceq_s y$) then $\preceq_x$ has the required properties, and we are done in this case.

Now suppose $\mathop{ht}(X)=\alpha+1$, a successor.
In this case $X^{(\alpha+1)}$ is a closed discrete subspace of $X$, and contains the points in $X$ of scattered height $\alpha+1$.
Again, let $\mathcal{W}=\{W_x : x \in X\}$. This is an open cover of $X$. As $X$ is ultraparacompact  there is a clopen partition refining it, and so, after a little tidying, there is a clopen partition of $X$, say $X=\bigoplus \{X_s : s \in S\}$, where each $X_s$ is hereditarily paracompact, and scattered with exactly one point of scattered height  $\alpha+1$ in $X_s$ and in $X$.
If there is a partial order $\preceq_s$ on each $X_s$ as in the claim, then their union, $\preceq_X$, is as claimed for $X$.

So we can suppose $X^{(\alpha+1)}$ has exactly one point $\ast$. 
Let $\mathcal{W}=\{W_x : x \in X \setminus \{\ast\}$. This is an open cover of $X\setminus \{\ast\}$ by sets each with scattered height strictly less than $\alpha+1$. 
As $X$ is hereditarily ultraparacompact there is a clopen partition of $X\setminus \{\ast\}$, say $X\setminus\{\ast\}=\bigoplus \{X_s : s \in S\}$, where each $X_s$ is hereditarily paracompact, and scattered of height strictly less than $X$.
Inductively, then, there are partial orders, $\preceq_{X_s}=\preceq_s$ on each $X_s$ as in the claim (down-sets open and contained in the canonical neighborhood,  up-sets finite).
Define $\preceq_X$ to be the disjoint union of all the $\preceq_s$ and then set $\ast$ to be the $\preceq_X$ maximum ($x \preceq_X \ast$ for all $x$  in $X$).
This is a partial order. Down-sets are open and contained in the canonical neighborhood.
And  up-sets are finite, indeed for an $x$ in $X_s$, the  up-set in $(X,\preceq_X)$ of $x$ has exactly one more element, namely $\ast$, than the  up-set of $x$ in $(X_s,\preceq_s)$, which is finite.
Thus $\preceq_X$ is as required.
\end{proof}

\begin{thr}\label{th:hpcpt_sc}
Let $X$ be hereditarily paracompact and scattered. Then $\square X^\om$ and $\nabla X^\om$ are hereditarily D.
\end{thr}
\begin{proof}
Fix a partial order $\preceq$ on $X$ as in the preceding lemma (down-sets open and contained in the canonical neighborhood,  up-sets finite). 
Extend over $\square X^\om$ pointwise and over $\nabla X^\om$ pointwise mod finite.

The box product of open sets is open, and $\down{x} = \square_n \down{x(n)}$, so down-sets for $\preceq$ on $\square X^\om$ are open, and hence $\preceq$ is topological.
Similarly, down-sets for $\preceq$ on $\nabla X^\om$ are open, and again $\preceq$ is topological. 
The box product   of finite (hence discrete) sets is discrete, and $\up{x} = \square_n \up{x(n)}$, so  up-sets for $\preceq$ on $\square X^\om$ are  discrete.
Similarly,  up-sets for $\preceq$ on $\nabla X^\om$ are discrete.  
Hence the claim follows from Proposition~\ref{pr:order_implies_hD}.
\end{proof}

\section{Scattered with Finite or Countable Height}\label{sec:fsct}

\subsection{Bounded Finite Scattered Height}

In \cite{peng2008hjm} Peng showed that the box product of scattered spaces with bounded finite scattered height is $D$. 
For the same spaces, using a simple and natural application of our `topological order' machinery, we show that the box and nabla products are hereditarily $D$. 

Let $X$ be scattered. 
As a trivial, but useful observation, let us note: if we add one isolated point to $X$, to get $X^*$, then $X$ and $X^*$ have the same scattered height.
We can put a partial  order, $\preceq$, on $X$ by $x \preceq y$ if and only if $x=y$ or $ht(x,X) < ht(y,X)$. Observe that down-sets in this order are open (they are, indeed, the canonical neighborhoods). 
Extend $\preceq$ over $\square X^\om$ pointwise, and over $\nabla X^\om$ pointwise mod finite. Again observe that down-sets are open, and hence these partial orders are topological.

\begin{prop} \label{pr:bd_finite_ht}
Let $(X_n)_n$ be scattered spaces with scattered height bounded by finite $m$. Then $\square_n X_n$ and $\nabla_n X_n$ are hereditarily D.
\end{prop}
\begin{proof} Let $X=\bigoplus_n X_n$. Then $X$ has finite scattered height no more than $m$. It suffices to show $\square X^\om$ and $\nabla X^\om$ are hereditarily $D$. 
\medskip

We start with the box product.
We prove this by induction on the scattered height of $X$. If $X$ is discrete then the claim is trivial. So assume all box products of scattered spaces of scattered height strictly less than that of $X$ are hereditarily $D$.

As $\preceq$ is topological, it suffices, by Proposition~\ref{pr:order_implies_hD}, to verify that all  up-sets are hereditarily $D$.
To this end, take any $x$ in $\square X^\om$. Observe that $\up{x} = \square_n \up{x(n)}$.  
If $x(n)$ is isolated then $\up{x(n)} = \{x(n)\} \cup X'$ (where $X'=X \setminus \mathop{I}(X)$). 
And as  noted above, $\up{x(n)}$ has the same scattered height as $X'$. Otherwise, $x(n)$ being a limit point implies $\up{x(n)}$ is contained  in $X'$.
\emph{Either way} the scattered height of $\up{x(n)}$ is no more than that of $X'$, and so strictly less than the scattered height of $X$ (use finite scattered height, here).
Hence by the inductive hypothesis, $\up{x}$ is hereditarily $D$. 
\medskip

The proof for the nabla product is almost identical.
Note that $\up{x} = \nabla_n {\up{x(n)}}$, so the same argument as above shows that the    up-sets are hereditarily $D$.
\end{proof}

\subsection{Unbounded Finite, and Countable Scattered Height}

The result above, Proposition~\ref{pr:bd_finite_ht}, does not extend to nabla products of space of countable, or even, finite (but \emph{unbounded}) scattered height.

\begin{exa} There is a space $X$ of countable scattered height such that $\nabla X^\om$ is not $D$. 
There is a sequence $(X_n)_n$ of spaces where $X_n$ is scattered of height $n$, but $\nabla_n X_n$ is not $D$.
\end{exa}
\begin{proof} 
By Lemma~\ref{l:embed_in_nabla}, it suffices to construct a $P$-space $X$ which is the increasing union of subspaces $X_n$ (so $X_\delta=X$ and $X$ embeds as a closed subspace in $\nabla_n X_n$), where each $X_n$ is scattered of height $n$, but $X$ is not $D$.

The space $X$ is obtained by invoking the machine from Theorem~\ref{thr:ex_machine} to the `base' space,  $B$ which is $\{0,1\}^{\beth_{\omega_1}}$ with $<\beth_{\omega_1}$ boxes (recall the definition of $\beth_{\lambda +1} = |\mathcal{P}(\beth_\lambda)|$, and $\beth_\lambda = \sup \{ \beth_\alpha : \alpha < \lambda\}$ for $\lambda$ limit).
In other words, basic neighborhoods of $x$ in $B$ have the form $N(x,\alpha)=\{ y \in B : y \restriction \beth_\alpha = x \restriction \beth_\alpha\}$.
Then $B$ is $\om_1$-metrizable, but not metrizable. It has cardinality $2^{\beth_{\omega_1}} = \beth_{\om_1+1}$ but, critically, has weight $\beth_{\om_1}$.
Hence, $2^{w(B)} \le |B|$, as required. 
\end{proof}

The situation with box products is not quite so clear cut. If $Y$ is any scattered space of countable height which is not $D$ -- for example, the space $X$ of the previous example, but also the space $\Gamma$ from \cite{van1977another} which inspired $X$ -- then $\square Y^\om$ is not $D$. However the authors do not know the answer to:

\begin{ques}
Is there a box product of spaces of finite (unbounded) scattered height which is not  $D$?
\end{ques}


\section{Small Weight and Character; Metrizable}

\subsection{Small weight and character}\label{ssec:small}
Under $\mathfrak{d}=\mathfrak{c}$  and the Model Hypothesis (MH${}_\mathfrak{c}$, stated below), the nabla and box products of compact, first countable spaces are paracompact (see \cite{roitman2011paracompactness, roitman2015paracompactness}). 
We now show that in both cases the nabla and box products are $D$.
The box product result follows immediately from Theorem~\ref{th:boxD_iff_nablaD} and the corresponding nabla result. 
While the nabla theorems are a consequence of the set theoretic hypotheses implying a certain structure on the nabla product, which -- the next general result says -- implies the hereditary $D$ property.

\begin{lemma}\label{l:inc_union_discrete_is_D}
Let $Y$ be the union of subspaces $Y_\alpha$, $Y_\alpha \subseteq Y_{\alpha +1}$, for $\alpha < \lambda$. Suppose that for each $y$ in $Y_\alpha \setminus \bigcup_{\beta<\alpha} Y_\beta$, there is an open set $W_y$ such that $W_y \cap Y_\alpha = \{y\}$.
Then $Y$ is hereditarily $D$.

\end{lemma}
\begin{proof}
First let us note that the hypothesis imply that each $Y_\alpha$ is closed in $Y$.

Now, clearly any subspace $Y'$ of $Y$ has the same structure: let $Y'_\alpha=Y_\alpha \cap Y'$ and $W'_y = W_y \cap Y'$. So to show $Y$ is hereditarily $D$ it suffices to show it is $D$.

Take any neighbornet $U$ on $Y$ such that for every $y$ in $Y$ we have $U(y) \subseteq W_y$. 
Then let $D_0=Y_0$ and $U_0=U(D_0)$. Note $D_0$ is closed discrete, while $U_0$ is open and contains $Y_0$.
Recursively, let $D_\alpha = (\bigcup_{\beta < \alpha} D_\beta) \cup (Y_\alpha \setminus \bigcup_{\beta <\alpha} U_\beta)$ and $U_\alpha= \bigcup_{\beta <\alpha} U_\beta \cup U(D_\alpha)$. 
It is easy to check that $D_\alpha$ is closed discrete and $U_\alpha$ is open and contains $Y_\alpha$. Then $D=\bigcup_{\alpha < \lambda} D_\alpha$ is closed discrete, and $U(D)=Y$.
\end{proof}

\begin{prop} Let $\nabla=\nabla_n X_n$ be a nabla product of first countable spaces, $X_n$.
Every subspace of $\nabla$ of size no more than $\mathfrak{d}$ is hereditarily $D$.
Hence, if $\mathfrak{d}=\mathfrak{c}$ and each $X_n$ has size no more than $\mathfrak{c}$, then $\nabla$ is hereditarily $D$. So, if $\mathfrak{d}=\mathfrak{c}$ then the box product of compact, first countable spaces is $D$.
\end{prop}
\begin{proof}
Take any subspace $Y$ of $\nabla$ with $|Y| \le \mathfrak{d}$, and enumerate it, $Y=\{y_\alpha : \alpha < \mathfrak{d}\}$. 
Let $Y_\alpha=\{y_\beta : \beta \le \alpha\}$.
Since every subset of $\nabla$ of size $<\mathfrak{d}$ is closed, for each $y=y_\alpha$, the set $W_y=\{ y_\beta : \beta \ge \alpha\}$ is open. Now apply Lemma~\ref{l:inc_union_discrete_is_D}. 
\end{proof}

\begin{de}[Roitman \cite{roitman2011paracompactness}]\label{MH}
The \emph{Model Hypothesis}, abbreviated $\mbox{MH}_\mathfrak{c}$, is the following statement: For some $\kappa$, $H(\mathfrak{c})$ is the increasing union of $H_\alpha$'s, for $\alpha < \kappa$, where each $H_\alpha$ is an elementary submodel of $(H(\mathfrak{c}), \in )$, and each $H_\alpha \cap \baire$ is not $\leq^*$-cofinal.
\end{de}

The statement $\mbox{MH}_\mathfrak{c}$ was implicitly defined in \cite{roitman1979more}. Roitman showed there that $\mbox{MH}_\mathfrak{c}$ follows from $\mathfrak{d=c}$ (hence from Martin's Axiom). It also holds in any forcing extension by uncountably many Cohen reals over a model of $\mbox{ZFC}$.

\begin{prop}
If the Model Hypothesis holds then the nabla product of first countable spaces of size $\le \mathfrak{c}$ is hereditarily $D$, and so the box product of compact, first countable spaces is $D$.
\end{prop}
\begin{proof} Let $(X_n)_n$ be first countable spaces, all of cardinality no more than the continuum, $\mathfrak{c}$. Let $\nabla=\nabla_n X_n$.
We assume the Model Hypothesis holds. 
Note that $H(\mathfrak{c})$ contains any first countable space of size no more than $\mathfrak{c}$, hence it contains $\nabla$.

Let $\nabla_\alpha=\nabla \cap H_\alpha$, for $\alpha < \kappa$. So $\nabla$ is the increasing union of the $\nabla_\alpha$'s. 
For any $z$ in $\nabla$ define $\alpha(z)$ to be the minimal $\alpha$ such that $z$ is in $\nabla_\alpha$. 
Then in Proposition~6.5 of \cite{roitman2011paracompactness} Roitman shows that each $\nabla_\alpha$ is (1) `strongly separated' (there is a discrete open collection $U = \{u_y : y \in \nabla_\alpha \}$ with each $y \in u_y$ and if $y \neq y'$ then $u_y = u_y'$) and (2) closed in $\nabla \setminus \bigcup_{\beta < \alpha} \nabla_\beta$. (More precisely she shows this if each first countable, $X_n$, is compact. But her only use of compactness is to deduce that the size of $X_n$ is no more than $\mathfrak{c}$.)
But from (1) and (2) we have that for every $z$ in $\nabla$ there is an open set $W_z$  such that $W_z \cap \nabla_{\alpha(z)} =\{z\}$. 
Thus Lemma~\ref{l:inc_union_discrete_is_D} applies, and $\nabla$ is hereditarily $D$.
\end{proof}

Consistently, we can remove the restriction to first countable spaces, provided we demand the weight is no more than $\omega_1$.
Indeed, in \cite{williams1984box} Williams showed under $\mathfrak{d}=\om_1$, that every $\nabla X^\om$, where $X$ is compact and of weight $\le \om_1$, is $\om_1$-metrizable, and hence hereditarily $D$. 
Applying Theorem~\ref{th:boxD_iff_nablaD} and recalling that every space of weight $\le \om_1$ embeds in a compact space of the same weight, we deduce:
\begin{thr} Assuming $\mathfrak{d}=\om_1$, if $X$ has weight $\le \om_1$ then $\nabla X^\om$ is  hereditarily $D$, and if in addition $X$ is compact, then $\square X^\om$ is $D$.
\end{thr}

Naturally we ask if this result holds in ZFC.
\begin{ques}
If $X$ has weight no more than $\omega_1$ then is $\nabla X^\omega$ (really) hereditarily $D$? 
Is at least true in ZFC that if $X$ is compact and $w(X) \le \omega_1$ then $\square X^\omega$ is $D$?
\end{ques}

In this direction we prove in ZFC that $\nabla (\omega_1 +1)^\omega$ is hereditarily $D$. 
(Recall we referred to this fact in the preliminaries.)  
To this end we will use (upgraded to the hereditary version) the next result by Guo and Junnila.

\begin{thr}[\cite{Guo&Junnila}]\label{t:Guo&Junnila}
Suppose $X = \bigcup_{\alpha < \kappa} X_\alpha$, where each $X_\alpha$ is (hereditarily) $D$, and for each $\beta < \kappa$, $\bigcup_{\alpha < \beta} X_\alpha$ is closed. Then $X$ is (hereditarily) $D$.
\end{thr}

\begin{exa}\label{ex:om1+1}
The space $\nabla (\omega_1+1)^\omega$ is hereditarily $D$.
\end{exa}
\begin{proof}
For an ordinal $\alpha < \omega_1$, let $X_\alpha = \{ x\in \nabla (\omega_1+1)^\omega: \forall^\infty n\in \omega, \ x(n) \leq \alpha \text{ or } x(n)= \omega_1 \}$. Let's note that each $X_\alpha$ is closed: if $y\notin X_\alpha$, there is an infinite set $N$ with $\alpha < y(n) < \omega_1$, for all $n\in N$. Hence, $\nabla_{n\in N} (\alpha, y(n)] \times \nabla (\omega_1 +1)^{\omega\setminus N}$ is disjoint from $X_\alpha$. 
Now, observe that each $X_\alpha$ is homeomorphic to $\nabla (\alpha \cup \{ *\})^\omega$, where $*$ is an isolated point. By Proposition \ref{th:hpcpt_sc} , $X_\alpha$ is hereditarily $D$, for each $\alpha < \omega_1$. 
Finally, since $\nabla (\omega_1+1)^\omega$ is a $P$-space (countable unions of closed sets are closed), $X_\alpha$ is hereditarily $D$ and $\nabla (\omega_1+1)^\omega = \bigcup_{\alpha< \omega_1} X_\alpha$, Theorem \ref{t:Guo&Junnila} applies. 
\end{proof}

\subsection{Metrizable}\label{ssec:metric}
Our primary objective in this section is to prove the following.
\begin{thr}\label{th:nablametricD} Let $X$ be a metrizable space. 
Then $\nabla X^\om$ is  hereditarily D. If $X$ has weight no more than $\mathfrak{d}$ then $\square X^\om$ is hereditarily D.
\end{thr}

For the proof we use some machinery due to Gruenhage \cite{gruenhage2006note}.
Let $Y$ be a space. A relation $R$ on $Y$ is
nearly good if $x \in \overline{A}$ implies $xRy$ for some $y \in A$.
For a neighborhood assignment, $U$, on $Y$, and subsets $Y'$ and $D$ of $Y$, we say
$D$ is $U$-sticky mod $R$ on $Y'$ if whenever $x \in Y'$ and $xRy$ for some $y \in D$, then $x \in U[D]$. 
We say more briefly that $D$ is $U$-sticky mod $R$ if $D$ is $U$-sticky mod $R$ on $Y$.
Given a neighborhood assignment $U$ on $Y$,  a subset $Z$ of $Y$ is $U$-close
if whenever $x, x'$ are in  $Z$ we have  $x \in U(x')$ (equivalently, $Z \subseteq U(x)$ for every $x$ in $Z$). 
Recall that the \emph{extent} of $(Y)$ is $e(Y) = \sup \{ |C|: C\subseteq Y \text{ is closed discrete}\}$.

It follows immediately from the definitions:
\begin{lemma}\label{l:gr2.0}
Let $U$ be a neighborhood assignment on $Y$, and $R$ a nearly good
relation. If $D$ is $U$-sticky mod $R$ on $Y'$ , then $\overline{D}\cap Y' \subseteq U[D]$.
\end{lemma}

The following result is a direct extension of Proposition 2.3 in \cite{gruenhage2006note}, and the proof is left to the reader. It is needed to prove  Proposition \ref{PropFibersofU-closedSets} which uses elementary submodels as per Gruenhage.  

\begin{prop}
Let $U$ be a neighborhood assignment on $Y$, let $\kappa$ be a cardinal, and let $R$ be a nearly good relation on $Y$. 
Suppose that given any closed discrete $D$ of size $\leq \kappa$ and non-empty closed $F \subseteq Y \setminus U (D)$ such that $D$ is $U$-sticky mod $R$ on $F$, there is a non-empty closed discrete $E \subseteq F$ of size $\leq \kappa$ such that $D \cup E$ is $U$-sticky mod $R$ on $F$. 
Then there is a closed discrete $D'$ in $Y$ with $U (D') = Y$.
\end{prop}

\begin{prop}\label{PropFibersofU-closedSets}
Let $\kappa$ be a cardinal.
Suppose $Y$ is a space which can be written as a disjoint union $Y= \bigcup_{i \in I} Y_i$, such that (1) each $Y_i$ is a closed $D$-subspace, (2) each $e(Y_i)\leq \kappa$, and (3) if $D_i \subseteq Y_i$ is closed discrete, for $i\in I$, and $J \subseteq I$ has size $< \kappa$, then $\bigcup_{i\in J} D_i$ is closed discrete. 
 
Let $U$ be a neighborhood assignment on $Y$ for which  there is a nearly good $R$ on $Y$ such that for any $y \in Y$, $R^{-1}(y)$ is the union of $\kappa$-many $U$-close sets.
 Then there is a closed discrete $D$ such that $U(D) = Y$.
\end{prop}

\begin{proof}
By the preceding proposition, we need only show that if $D$ has size $\leq \kappa$, is closed discrete, and $U$-sticky mod $R$ on some non-empty closed $F \subseteq Y \setminus U (D)$, then there is a non-empty closed discrete $E \subseteq F$ of size $\leq \kappa$ such that $D \cup E$ is $U$-sticky mod $R$ on $F$. For $i\in I$, let $F_i = F\cap Y_i$.

For $y \in Y$, let $R^{-1} (y) \setminus U (y) = \bigcup_{\alpha\in e(Y)} G_\alpha (y)$, where each $G_\alpha (y)$ is $U$-close.
Put all relevant objects in an elementary submodel $M$ of size $\kappa$ ($R,Y, U, D \in M$ and $D \subseteq M$). 
Let $<_M$ well-order $M$ in type $\kappa$. 
For each $i\in I$, choose a closed discrete set $E_0 \subseteq F_{i_0} \cap M$ of size $\leq \kappa$ (hypothesis (2)) such that $F_{i_0} \subseteq U(E_0)$ (hypothesis (1)). Suppose defined $E_\beta$, for $\beta < \alpha$, and look at
$$F^\alpha = \{x \in F \setminus U (\bigcup_{\beta < \alpha} E_\beta) : xRy \text{ for some } y \in D \cup \bigcup_{\beta < \alpha} E_\beta \}.$$

If $x \in F^\alpha$, then $x\in G_{\gamma} (y) \subseteq U(x)$ for some $y\in D \cup \bigcup_{\beta < \alpha} E_\beta$, and some $\gamma$. Moreover, if $x \in F^\alpha \cap M$,  then both $G_{\gamma} (y)$ and $U(x)$ are elements of $M$ since $U, R, y\in M$. 
Choose $e_\alpha \in F^\alpha \cap M$ such that the corresponding $G_{\gamma_{\alpha}} (y_{\alpha})$ is $<_M$-least possible not used so far. Note that $e_\alpha$ is in $Y_{i_\alpha}$, for some $i_\alpha \in I$. Choose a closed discrete set $E_\alpha \subseteq F_{i_\alpha} \cap M$ of size $\leq \kappa$ such that $F_{i_\alpha} \subseteq U(E_\alpha)$.
\\
If $F^\alpha = \emptyset$ for any $\alpha < \kappa$, then $D \cup \bigcup_{\beta < \alpha} E_\beta$ is closed discrete by hypothesis (3), and it is $U$-sticky mod $R$ relative to $F$, this would conclude the proof.
Otherwise, assume that $F^\alpha  \neq \emptyset$ for all $\alpha < \kappa$. Define $E = D \cup \bigcup_{\alpha < \kappa} E_\alpha$. 
We show that $D \cup E$ is closed discrete, and $U$-sticky mod $R$ on $F$. 
Clearly $E$ is relatively closed discrete in $U(E)$ by construction, so if we prove $D \cup E$ is $U$-sticky mod $R$ on $F$, then Lemma~\ref{l:gr2.0} applies, and as a consequence $E$ is closed discrete in $F$. 
To this end, suppose $x \in F \setminus U(D \cup E)$ and $xR y$ for some  $y_0 \in D \cup E$.
Then for all sufficiently large $\alpha$, we have $x \in F^\alpha$.
Let $\alpha_0$ be such that $x\in G_{\alpha_0} (y_0)$, and note that $G_{\alpha_0}(y_0) \in M$. Since the $U(e_\alpha)$'s always contain the $<_M$-least $G_\alpha(y)$, at some step $\gamma$ it was picked an $e_\gamma$ such that $U(e_\gamma) \supseteq G_{\alpha_0} (y_0)$ which puts $x \in U (E_{\alpha_0})$, contradiction.
\end{proof}

The hypothesis (3) of Proposition~\ref{PropFibersofU-closedSets}  
is naturally  satisfied hereditarily by box and nabla products of first countable spaces, with $\kappa=\mathfrak{d}$.

\begin{lemma}\label{l:hyp3} Let $X$ be a first countable space. Let $q$ be the quotient map of $\square=\square X^\om$ to $\nabla=\nabla X^\om$. Let $\kappa=\mathfrak{d}$.

(a) If $Y$ is a subspace of $\nabla$ then $Y=\bigcup_{x \in Y} \{x\}$ satisfies (3) of Proposition~\ref{PropFibersofU-closedSets}.

(b) If $Y$ is a subspace of $\square$ then $Y=\bigcup_{x \in \nabla} \left(q^{-1}(x) \cap Y\right)$ satisfies (3) of Proposition~\ref{PropFibersofU-closedSets}.
\end{lemma}
\begin{proof}
Claim (a) is immediate from the fact that any subset of size $<\mathfrak{d}$ of a nabla product of first countable spaces is closed and discrete (see Proposition 5.2.4 in \cite{roitman2011paracompactness}). 

Towards (b),  suppose $D_x$ is a closed discrete subset of $q^{-1}(x) \cap Y$, for $x$ in $J \subseteq \nabla$, where $|J| < \mathfrak{d}$. 
Then noting, as above, that $J$ is closed and discrete in $\nabla$, it is easy to combine, using continuity of $q$,  witnesses of discreteness of $J$ and each $D_x$, to see that $\bigcup_{x \in J} D_x$ is closed discrete in $Y$, as required for (3) of Proposition~\ref{PropFibersofU-closedSets}.
\end{proof}

Now we prove Theorem~\ref{th:nablametricD}.

\begin{proof}
Let $X$ be a metrizable space, say with compatible metric $d$. Write  $\square$ for $\square X^\om$ and $\nabla$ for $\nabla X^\om$. We prove the result for $\square$ (where we additionally assume that the weight of $X$ is $\le \mathfrak{d}$), then at the end we will indicate the small changes needed to establish the claim for $\nabla$ (where there is no weight restriction).

Take any subspace $Y$ of $\square$. To show $Y$ is hereditarily D we will apply Proposition~\ref{PropFibersofU-closedSets}. 
To do so we first verify that $Y$ satisfies the hypothesis on the space, and then show that for every neighbornet $U_Y$ on $Y$ there is a nearly good relation, $R_Y$ as specified in Proposition~\ref{PropFibersofU-closedSets}.

Recall that we denote by $q$ the quotient map from $\square$ to $\nabla$.
The fibres, $q^{-1}(x)$, of $q$, as the $\sigma$-product of metrizable spaces with the box product topology, are stratifiable (van Dowuen~\cite{vD_not_normal}) and hence hereditarily D.
In particular,  $Y$ is the union of the closed sets, $Y_x=q^{-1}(x) \cap Y$, which are D, as $x$ runs over $\nabla$.
Let $\kappa=\mathfrak{d}$. 
Since $X$ has weight no more than $\mathfrak{d}$, 
each $Y_x$  has net weight $\le \mathfrak{d}$, and so $e(Y_x) \le \mathfrak{d}$. 
Combined with Lemma~\ref{l:hyp3}, we see $Y$ satisfies the topological hypotheses (1), (2) and~(3) of Proposition~\ref{PropFibersofU-closedSets}.

Take any neighbornet $U_Y$ on $Y$.
Pick a neighbornet $U$ on $\square$ so that for every $y$ in $Y$ we have $U_Y(y)=U(y) \cap Y$. 
Suppose we have  found a nearly good $R$ on $\square$  such that for any $y \in \square$, $R^{-1} (y) \setminus U (y)$ is the union of $\kappa$-many $U$-close sets. 
Then $R_Y$, which is $R$ restricted to $Y$, witnesses the same property for $U_Y$.
So to complete the argument it suffices to find $R$ for a neighbornet $U$ on $\square$.

Each $x$ in $\square$ has basic neighborhoods, $N(x,f)=\{y : \forall n \ d(x(n),y(n))< 2^{-f(n)}\}$. 
Fix a dominating set, $D$, in $(\om^\om, \le)$, of minimal size, $\mathfrak{d}$, closed under finite modifications and such that if $f$ is in $D$ then so is $f+1$ (where $(f+1)(n)=f(n)+1$). 
Let $\mathcal{B}_{x} = \{ N(x,f) : f \in D\}$.
This is a neighborhood base for $x$.

Let $U$ be a neighbornet in $\square$. Define a relation $R$ on $\square$ by $xRy$ if and only if there is an $f$ in $D$ such that $x \in N(y,f) \subseteq U(x)$. 
We check $R$ is nearly good. Take any subset $A$ of $\square$ and $x$ in $\overline{A}$. Pick $g$ in $D$ so that $x \in N(x,g) \subseteq U(x)$. Let $f=g+1$, and note $f$ is in $D$. Take any $y \in A \cap N(x,f)$.
Then $y$ is in $A$, and $x \in N(y,f) \subseteq N(x,g) \subseteq U(x)$ (because the inclusions hold coordinatewise, as $d$ is symmetric and satisfies the triangle inequality).
In other words, $xRy$, as required.
Finally, for each $f$ in $D$, let $C(f) = \{x : x \in N(y,f)  \subseteq U(x)\}$. Then $C(f)$ is $U$-close, and
$R^{-1} (y) = \bigcup \{C(f) : f \in D\}$.

\medskip

For $\nabla$, the argument is very similar, and a little simpler. 
Take any subspace $Y$ of $\nabla$. 
For $Y_x$ just take $\{x\}$, as $x$ runs through $Y$. 
Clearly the extent restriction on the $Y_x$'s is satisfied (without limiting the weight of $X$). 
Now Proposition~\ref{PropFibersofU-closedSets} applies to $Y$, and we see that $\nabla$ is hereditarily $D$. 
Basic neighborhoods, $N(x,f)$, in $\nabla$ are the natural mod finite modifications of those in $\square$, so define the relation $R$ in the same way as above ($xRy$ if and only if there is an $f$ in $D$ such that $x \in N(y,f) \subseteq U(x)$)  and repeat the argument (mod finite on coordinates) to see it has the requisite properties.
\end{proof}

When our space is first countable  and has small weight - in particular, of weight no more than $\omega_1$ - we can adapt the argument for metrizable spaces as follows.

\begin{thr}
Let $X$ be a first countable space with  weight no more than $\mathfrak{d}$ and strictly less than $\aleph_\om$. 
Then $\nabla X^\om$ is hereditarily $D$.
\end{thr}
\begin{proof} 
We will apply Proposition~\ref{PropFibersofU-closedSets} to $\nabla X^\om$ with $\kappa=\mathfrak{d}$ and the partition all the singletons. From Lemma~\ref{l:hyp3} the hypotheses are clearly satisfied. 
 We show that $uw(\nabla X^\om) \le \mathfrak{d}$, and so $\nabla X^\om$ has a $\mathfrak{d}$-point network.  
As $X$ is a subspace of $Y=I^{\mu}$, where $\mu$ is the weight of $X$, we see $\nabla X^\om$ is a subspace of $\nabla Y^\om$.
So it suffices to find a base for a compatible uniformity of $\nabla Y^\om$ of size no more than $\mathfrak{d}$.  
Further, the standard quotient map of $\square Y^\om$ to $\nabla Y^\om$ carries compatible uniformities, and their bases, to compatible uniformities, and their bases. Thus we have reduced our problem to showing that $\square Y^\om$ has a compatible uniformity with a base of size no more than $\mathfrak{d}$.

A basis for a compatible uniformity, $\mathcal{D},$ on $Y$ is given by  $D(F)=\{(y,z) \in Y^2 : |z(\alpha)-y(\alpha)| < 2^{-|F|} \ \forall \, \alpha \in F\}$, where  $F \in [\mu]^{<\om}$.
Note that $\mathop{cof}{\mathcal{D}} \le \mathop{cof}{[\mu]^{<\om}}$. 
A basis for a compatible uniformity on $\square Y^\om$ is given simply by taking countable products from $\mathcal{D}$. 
Hence this compatible uniformity on $\square Y^\om$ has a base of size no more than the cofinality of $([\mu]^{<\om})^\om$. 
However, from \cite{gartside2016tukey} we know that for every $n$ in $\om$ we have that $([\aleph_n]^{<\om})^\om$ is Tukey equivalent to $\om^\om \times [\aleph_n]^{<\om}$, and so its cofinality is $\max( \mathfrak{d},\aleph_n)$. 
By hypothesis, $\mu$ is $\le \mathfrak{d}$ and $<\aleph_\om$, so $\mathop{cof}(([\mu]^{<\om})^\om)=\max( \mathfrak{d},\mu)=\mathfrak{d}$, as required. 
\end{proof}

\bibliographystyle{plain}
\bibliography{biblio}

\end{document}